\documentclass{amsart}

\usepackage{amsmath,amssymb}
\usepackage{amsfonts}
\usepackage{graphicx}
\usepackage{epstopdf}
\usepackage[loose]{subfigure}
\newtheorem{theorem}{Theorem}[section] 
\newtheorem{lemma}[theorem]{Lemma}     
\newtheorem{corollary}[theorem]{Corollary}
\newtheorem{proposition}[theorem]{Proposition}
\theoremstyle{plain}
\newtheorem{definition}{Definition}

\newcommand{\abs}[1]{\lvert#1\rvert}

\title[Digit sets for connected tiles via similar matrices I]
 {Digit sets for connected tiles via similar matrices I: Dilation matrices with rational eigenvalues} 

\author{Avra S.~Laarakker and Eva Curry}

\begin{document}

\begin{abstract}
Given any $m$-dimensional dilation matrix $A$ with rational eigenvalues, we demonstrate the existence of a digit set $D$ such that the attractor $T(A,D)$ of the iterated function system generated by $A$ and $D$ is connected.  We give an easily verified sufficient condition on $A$ for a specific digit set, which we call the centered canonical digit set for $A$, to give rise to a connected attractor $T(A,D)$.
\end{abstract}

\maketitle

\section{Introduction}

A \emph{dilation matrix} is a matrix $A \in M_{m}(\mathbb{Z})$ that is expanding in the sense that all eigenvalues $\lambda$ of $A$ satisfy $\abs{\lambda} > 1$.  Let $A$ be a dilation matrix.  Let $D$ be a complete set of coset representatives of $\mathbb{Z}^m / A(\mathbb{Z}^m)$, with exactly one representative from each coset.  Note that the number of digits is equal to  $\abs{A}$ \cite{gm1992}.  The set $D$ is called a \emph{digit set} (or, more specifically, a \emph{basic digit set}\cite{m1982}) for $A$, and the elements $d \in D$ are called \emph{digits}.  If $F$ is congruent to $\mathbb{R}^m/\mathbb{Z}^m$, then $A(F) \cap \mathbb{Z}^m$ is a (basic) digit set for $A$ \cite{me2006}.  For example, if $F = [0,1)^m$, the standard fundamental domain for the lattice $\mathbb{Z}^m$, then $A(F) \cap \mathbb{Z}^m$ corresponds to the usual base $b$ digit set for $b \geq 2$ a dilation in one dimension.  When $F$ is a translated fundamental domain centered at the origin (which we will refer to as the ``centered fundamental domain''), $\displaystyle{F = \left( -\frac{1}{2}, \frac{1}{2} \right]^m}$, we will call $D_{A} := A(F) \cap \mathbb{Z}^m$ the \emph{centered canonical digit set for $A$}.

Let $A$ be a dilation matrix and $D$ any digit set for $A$.  For each $d \in D$, set 
\[ f_{d}(x) := A^{-1}(x+d), \qquad x \in \mathbb{R}^m. \]
The collection of maps $\{ f_{d}:\ d \in D \}$ is an iterated function system (IFS).  Let $T(A,D)$ be the attractor of the IFS.  Then \cite{gm1992} 
\[ T(A,D) = \{ \sum_{j=1}^{\infty} A^{-j}d_{j}:\ d_{j} \in D \}. \]
Note that $T(A,D)$ is self-affine, that is, 
\[ T(A,D) = \bigcup_{d \in D} A^{-1}(T(A,D) + d) = \bigcup_{d \in D} f_{d}(T(A,D)). \]
It is also known that $T(A,D)$ tiles $\mathbb{R}^m$ by translation either by $\mathbb{Z}^m$ or by a sub-lattice of $\mathbb{Z}^m$ \cite{lw1997}.  In fact, if all singular values $\sigma$ of $A$ satisfy $|\sigma|>2$, then there is a digit set $D$ such that $T(A,D)$ tiles $\mathbb{R}^m$ by translation by the full lattice $\mathbb{Z}^m$ \cite{me2006}.

In this paper, we consider dilation matrices $A$ having only rational eigenvalues, $\lambda \in \mathbb{Q}$.  We demonstrate the existence of a digit set $D^{*}_{A}$ such that $T(A,D^{*}_{A})$ is connected for any such dilation matrix $A$.  The digit set $D^{*}_{A}$ will be derived from the centered canonical digit set for the Jordan form $J$ of the matrix $A$, relying on the fact that $J$ is a similar matrix to $A$.  Our approach thus contrasts with that of Kirat and Lau \cite{kl2000}, who have studied connectedness of sets $T(A,D)$ by generalizing the usual digit sets for base $b$ number systems in one dimension to consecutive colinear digit sets, which are not required to be basic digit sets in general.

Since the question of connectedness for the sets $T(A,D)$ that we consider is of interest to researchers from diverse backgrounds, including in the wavelet and measureable dynamical systems communities as well as computational and algebraic number theory, we include a brief review of the Jordan form.  See \cite{hj1990} for a more complete discussion.
\begin{definition}
Two matrices $A$ and $B$ are \emph{similar} if there exists an invertible matrix $P$ such that $A = PBP^{-1}$.
\end{definition}

\begin{definition}
Let $A$ be any matrix in $M_{m}(\mathbb{C})$.  Let $\lambda_1, \ldots, \lambda_r$ be the eigenvalues of $A$, where each $\lambda_i$ corresponds to a distinct, irreducible eigenspace.
\begin{itemize}
\item For each eigenvalue $\lambda_i$ of $A$, the \emph{Jordan block} corresponding to $\lambda_i$, $J_{i}$, is defined as follows:
  \begin{itemize}
  \item if the generalized eigenspace of $\lambda_{i}$ is one-dimensional, then the Jordan block corresponding to $\lambda_{i}$ is also one-dimensional, with $J_{i} = \lambda_{i}$;
  \item if the generalized eigenspace of $\lambda_{i}$ has dimension $k > 1$, then the Jordan block $J_{i}$ is a \mbox{$k{\times}k$} matrix with $\lambda_{i}$ in every entry on the diagonal, $1$ in every entry on the superdiagonal, and $0$ in every other entry,
\[ J_{i} = \left[ \begin{array}{cccc} \lambda & 1 & & 0 \\ & \ddots & \ddots & \\ & & \ddots & 1 \\ 0 & & & \lambda \end{array} \right]. \]
  \end{itemize}

\item The \emph{Jordan form} of $A$ is the matrix $J := \textrm{diag}{\{J_{i}\}}$.  Note that $J$ is unique up to reordering of the eigenvalues $\lambda_1, \ldots, \lambda_r$.
\end{itemize}
\end{definition}

It can be shown that $A$ and it's Jordan form $J$ are similar matrices.  As well, the matrix $P$ such that $A = PJP^{-1}$ is the matrix whose columns are the eigenvectors and generalized eigenvectors for the eigenvalues of $A$, listed in the appropriate order.  When $A$ has entries in any subset of $\mathbb{C}$, even $\mathbb{Z}$, $\mathbb{Q}$, or $\mathbb{R}$, it may still have irrational or complex eigenvalues, with $J, P, P^{-1} \in M_{m}(\mathbb{C})$ in general.  However, in this paper we consider only matrices with rational eigenvalues.  Thus for our dilation matrices $A$, we will have $J \in M_{m}(\mathbb{Q})$.  It follows that $P, P^{-1} \in M_{m}(\mathbb{Q})$ as well.

\section{Methods}

\subsection{Using the Jordan Form}
Experimental evidence indicates that an attractor $T(A,D)$ for the iterated function system associated with a dilation matrix $A$ and centered canonical digit set $D$ fails to be connected primarily when $A$ includes too large of a skew component\cite{l2009}, \cite{kl2000}.  Thus our basic approach to finding digit sets for which $T(A,D)$ is connected is to consider the Jordan form of $A$, where as much of the skew component as possible has been removed.  We can generate digit sets for $A$ from digit sets for the Jordan form $J$, or any matrix that is similar to $A$, as follows.

Let $A$ be a dilation matrix, let $J$ be any matrix that is similar to $A$ (for example, the Jordan form of $A$), and let $P$ be any matrix that gives a similarity transformation $A = PJP^{-1}$.  Suppose that $D_J$ is a basic digit set for $J$, and write $D_J = \{ g_1, \ldots, g_q\}$, where $q = |\det(J)| = |\det(A)|$.  Since $P$ is an invertible matrix (and thus a continuous linear transformation), $D_A = PD_J$ will be a basic digit set for $A$, with digits $d_i = Pg_i$ for $i = 1, \ldots, q$.  Likewise, the attractors of the corresponding iterated function systems are related:
\begin{align*}
T(A,D_A) &= \sum_{i=1}^{\infty} A^i d_i \quad \mbox{for $d_i \in D_A$}\\
         &= \sum_{i=1}^{\infty} (PJP^{-1})^i d_i \\
         &= \sum_{i=1}^{\infty} PJ^iP^{-1} d_i \\
         &= P\left( \sum_{i=1}^{\infty} J^i(P^{-1}d_i) \right) \\
         &= P(T(J,D_J)).
\end{align*}
Since the mapping defined by $P$ is a continuous homeomorphism between $T(A,D_A)$ and $T(J,D_J)$, we may conclude the following.
\begin{lemma} The set $T(A,D_A)$ is connected if and only if the set $T(J,D_J)$ is connected.
\end{lemma}

While any constant multiple of $P$ will define a similarity transformation
between $A$ and $J$, $P$ is generally taken to have determinant $1$ for the
Jordan decomposition of a matrix $A$.  This means, however, that even if $J$
has integer entries, the entries in $P$ are not necessarily integral, and
thus the new digit set $D_A = PD_J$ will not necessarily be in
$\mathbb{Z}^m$.  For example, consider the matrix $A$ and its Jordan form
$J$,
\[ A = \left[ \begin{array}{cc} 3 & 10 \\ 0 & 3 \end{array} \right], \qquad J
= \left[ \begin{array}{cc} 3 & 1 \\ 0 & 3 \end{array} \right]. \]
Then the matrix $P_1$ with determinant $1$ that gives the similarity
transformation $A = P_1JP_1^{-1}$,
\[ P_1 = \left[ \begin{array}{cc} \sqrt{10} & \sqrt{10} \\ 0 & \frac{1}
{\sqrt{10}} \end{array} \right], \]
is not in $M_2(\mathbb{Z})$, nor is the resulting digit set $P_1D_J$.
However, $P = \sqrt{10} P_1 \in M_2(\mathbb{Z})$ also satisfies $A =
PJP^{-1}$, and will yield a digit set $D_A = PD_J$ that is in $\mathbb{Z}^2$.

To ensure that $D_A \subset \mathbb{Z}^m$, we will consider only matrices $J
\in M_{m}(\mathbb{Z})$ and will take the smallest multiple of $P_1$ such that
$P \in M_{m}(\mathbb{Z})$.  The columns of $P$ are eigenvectors and
generalized eigenvectors of $A$, satisfying $(A - \lambda I)^s\mathbf{v}$ for
some eigenvalue $\lambda$ of $A$ and multiplicity $s$ ($s \geq 2$ when
$\mathbf{v}$ is not itself an eigenvector), so we see that
\[ P = \min \{ rP_1:\ r > 0,\  rP_1 \in M_m(\mathbb{Z}) \} \]
is well-defined.  Notice that the measure of the set $T(A,D_A)$ is $\det{P} =
r^m$ times the measure of the set $T(J,D_J)$, so the new tiles $T(A,D_A)$
that we find from the Jordan form for $A$ will not have measure $1$ in
general.

We conjecture that the centered canonical digit set $D_J$ for $J$ will always
yield a connected attractor $T(J,D_J)$, and thus allow us to find a digit set
$D_A = PD_J$ for $A$ that gives a connected attractor $T(A,D_A)$.  In this
paper, we prove the result in the case that $A$ has rational eigenvalues.

\subsection{Tiling Considerations}
Let $A$ be a dilation matrix and $D$ a digit set for $A$.  By Theorem 1.1 of \cite{lw1997}, there is a lattice $\Gamma$ for which $T(A,D)$ gives a lattice tiling of $\mathbb{R}^m$.  That is, $\Gamma$ is a sublattice of $\mathbb{Z}^m$ satisfying
\begin{enumerate}
\item $T(A,D) \cap (T(A,D) + \gamma)$ has $m$-dimensional Lebesgue measure $0$ for all nonzero $\gamma \in \Gamma$, and
\item $\cup_{\gamma \in \Gamma} (T(A,D) + \gamma) = \mathbb{R}^m$.
\end{enumerate}
We call such a lattice $\Gamma$ a \emph{lattice of translations for $T(A,D)$}.

\begin{lemma}\label{lem_digittranslates}  There exists a lattice of translations $\Gamma_A$ for $T(A,D)$ with $D \subset \Gamma_A$.
\end{lemma}

\begin{proof}
If $\Gamma = \mathbb{Z}^m$, such as when $A$ yields a radix representation for $\mathbb{Z}^m$ with digit set $D$ \cite{me2006}, then the result is immediate.  In this case, $\Gamma_A$ is the unique lattice of translations for $T(A,D)$.

In general, let $\Gamma$ be any lattice of translations for $T(A,D)$.  Let $\{ b_1, \ldots, b_m \}$ be a lattice basis for $\Gamma$.  Then every $\gamma \in \Gamma$ can be written as a linear combination of the basis elements with integer coefficients, and
\[ \mathbb{R}^m = \bigcup_{c_1, \ldots, c_m \in \mathbb{Z}} \left( c_1 b_1 + \cdots + c_m b_m + T(A,D)\right). \]
Using the self-affine property of $T(A,D)$,
\begin{align*}
\mathbb{R}^m = A\mathbb{R}^m &= \bigcup_{c_1, \ldots, c_m \in \mathbb{Z}} \left( c_1 Ab_1 + \cdots + c_m A b_m + AT(A,D) \right) \\
  &= \bigcup_{c_1, \ldots, c_m \in \mathbb{Z}} \bigcup_{d \in D} \left( c_1 Ab_1 + \cdots + c_m Ab_m + d + T(A,D) \right).
\end{align*}
Thus $\Gamma_A = A\Gamma + D$ is also a lattice of translations for $T(A,D)$.  The lattice $\Gamma$ must contain $0$, therefore by construction $D \in \Gamma_A$.
\end{proof}

Note that $\Gamma = A\Gamma + D$ if and only if $\Gamma$ is $A$-invariant.  This implies that if $T(A,D)$ has a lattice of translations $\Gamma$ that is not $A$-invariant, then $\Gamma$ will not be unique.  Also, if $\Gamma$ is the unique lattice of translations for $T(A,D)$, then $\Gamma = A\Gamma + D$ must be $A$-invariant.   As a partial converse, if $\Gamma$ is $A$-invariant, then $\Gamma_A = \Gamma$ is the unique $A$-invariant lattice of translations for $T(A,D)$.  We cannot determine from the above discussion whether $T(A,D)$ may also have a non-$A$-invariant lattice of translations in this case, however.

A tile $T(A,D)$ need not have an $A$-invariant lattice of translations in general; Lagarias and Wang give an example of a matrix $A$ and digit set $D$ such that no lattice of translations $\Gamma$ for $T(A,D)$ can be $A$-invariant (\cite{lw1997}, equations (1.5)).  Inspection of the proof of their Theorem 1.1 shows that this can only occur for so-called ``stretched tiles'', however~\cite{lw1997}.  We conjecture that centered canonical digit sets do not give rise to stretched tiles.

For the remainder of the paper, we fix a lattice of translations $\Gamma_A$ with $D \subset \Gamma_A$.  In subsequent sections, we will also assume that $\Gamma_A$ is $A$-invariant.  The following lemma holds more generally, however.

\begin{lemma}\label{lem_DcapGamma}  If $D$ is the centered canonical digit set for $A$ and $F$ is the centered fundamental domain for $\Gamma_A$, then $D = AF \cap \Gamma_A$.
\end{lemma}
\begin{proof}
By definition,
\[ D = A\left( -\frac{1}{2},\frac{1}{2} \right]^{m} \bigcap \mathbb{Z}^m. \]
From Lemma~\ref{lem_digittranslates}, we know that $D \subset \Gamma_A$.  As well, $\Gamma_A \subseteq \mathbb{Z}^m$ implies that $(-1/2,1/2]^m \subseteq F$, so $D \subseteq AF \cap \Gamma_A$.  There are only $q = \abs{\det{A}}$ points in the set $AF \cap \Gamma_A$, however, which is equal to the cardinality of the digit set $D$.  Thus $D = AF \cap \Gamma_A$.
\end{proof}

The translates of $T = T(A,D)$ that are adjacent to the original tile $T$ play an important role in determining whether $T(A,D)$ is connected or disconnected.  We use the following definition from Scheicher and Thuswaldner \cite{st2002}.
\begin{definition}
For $s \in \Gamma_A$, let $B_s = T \cap (T+s)$.  The set of \emph{neighbours} of $T$ is the set
\[ S := \{ s \in \Gamma_A\backslash\{0\} : B_s \neq \emptyset \}. \]
\end{definition}

\subsection{Level Sets and the Iterated Approach to Connectedness}

To show connectedness of an attractor $T(A,D)$, we will take advantage of the iterated function system structure.  Kirat and Lau first proved the following useful result.
\begin{lemma}\cite{kl2000}\label{lem_limitconnected} Suppose that $T_n$ is a sequence of compact, connected subsets of $\mathbb{R}^m$, and that, in the Hausdorff metric, $T = \lim_{n \rightarrow \infty} T_n$.  Then $T$ is connected.
\end{lemma}
Note that this theorem holds for arbitrary sets $T$ and $T_n$ satisfying the hypotheses.  In our case, we will set $T_0 = [0,1]^m$ and recursively define
\[ T_n := A^{-1}\left( \bigcup_{d \in D} (T_{n-1} + d) \right). \]
We note that $T(A,D) = \lim_{n \rightarrow \infty} T_n$ in the Hausdorff metric \cite{br1993}.  We then see that the tile $T = T(A,D)$ is connected if and only if there exists an $N$ such that $T_n$ is connected for all $n > N$.

Kirat and Lau used Lemma~\ref{lem_limitconnected} to prove a criterion for connectedness, which we present a refinement of.  First, consider a finite subset $B$ of $\mathbb{Z}^m$.  The set $B$ will generate a sublattice of $\mathbb{Z}^m$, which we may consider as a graph.
\begin{definition}
We say that a set $S$ is \emph{$B$-connected} if $S$ forms a connected subgraph of the lattice generated by $B$.  If $\Gamma$ is any sublattice of $\mathbb{Z}^m$, we will also say that the set $S$ is \emph{$\Gamma$-connected} if $S$ forms a connected subgraph of $\Gamma$.
\end{definition}
We may now state the following criterion for connectedness of $T(A,D)$.
\begin{proposition}
Let $S$ be the set of neighbours of $T = T(A,D)$, and $B$ a basis of $\Gamma_A$ such that $B \subset S$.  If $D$ is $B$-connected, then $T$ is connected.
\end{proposition}
Our proof follows the basic method of Kirat and Lau.
\begin{proof}
We note that Lemma~\ref{lem_digittranslates} guarantees that $D \subset \Gamma_A$, so that the hypothesis that $D$ is $B$-connected makes sense.

To show that $T$ is connected, we first show that if any subset $Q$ of $\mathbb{R}^m$ that is congruent to $\mathbb{R}^m/\Gamma_A$ is connected, then $A^{-1}(Q+D)$ is connected.

Let $d, d' \in D$.  We want to show that there exists a sequence from $Q+d$ to $Q+d'$.  That is, a sequence $d = d_1, \ldots, d_r = d'$ such that $(Q+d_j) \cap (Q+d_{j+1}) \neq \emptyset$ for $j = 1, \ldots, r-1$.  This means that the path from $Q+d$ to $Q+d'$ is a path through neighbour translates (translates by elements of $S$) of our tile $T$.

We know that for any $d, d' \in D$ there exists a path $d = d_1, \ldots, d_r = d'$ such that $d_{j+1} - d_j \in B$ since $D$ is $B$-connected.  Also, $B \subset S$ implies that $Q \cap (Q+b) \neq \emptyset$ for all $b \in B$.  Thus $Q \cap (Q + d_{j+1} - d_j) \neq \emptyset$, and so $(Q + d_j) \cap (Q + d_{j+1}) \neq \emptyset$.  We see that this same sequence gives the path that we need to connect $Q+d$ and $Q + d'$.  Thus $Q+D$ is connected whenever $Q$ is connected, and therefore $A^{-1}(Q+D)$ is connected.

Let $T_0$ be the standard fundamental domain for $\Gamma_A$, a connected set, and let
\[T_{n+1} = A^{-1}(T+D) \quad \mbox{for $n \geq 0$} \]
as above.  By induction, with $Q = T_n$ for each $n \geq 0$, we see that each $T_n$ is connected.  Then, by Lemma~\ref{lem_limitconnected}, $T = \lim_{n \rightarrow \infty} T_n$ is connected as well.
\end{proof}
As the proposition hints at, even though out goal is to show connectedness of a set in $\mathbb{R}^m$, it is more practical to work in the discrete setting.  Thus instead of the approximating sets $T_n$ defined in the proof above, we wish to consider discrete sets $D_n$, the level sets of the digit set $D$.  Recall that the sets $D_n$ are defined as
\[ D_n = \{ k \in \mathbb{Z}^m: k = \sum_{i=0}^{n-1} A^i d_i,\ \mbox{with $d_i \in D$} \}, \]
and note that the sets can also be defined recursively:
\[ D_n := A(D_{n-1} + D), \quad D_{1} = D. \]
\begin{lemma}\label{lem_TnandDn} Suppose that $D_n \subset \Gamma_A$.  Let $D$ be the centered canonical digit set for $A$.  The set $T_n$ is connected with no finite cut sets if and only if the set $D_n$ is $\Gamma_A$-connected.
\end{lemma}
\begin{proof}
A closer look at the proof of Lemma~\ref{lem_digittranslates} reveals that, for any fixed integer $n \geq 1$, we can in fact find a lattice of translations $\Gamma_{A,n}$ for $T(A,D)$ such that $D_n \subset \Gamma_{A,n}$.  Note that in order to have $D_n \subset \Gamma_A$ for every $n \geq 1$, $\Gamma_A$ needs to be $A$-invariant, however.  The condition that $D_n \subset \Gamma_A$ is required for the hypothesis that $D_n$ is $\Gamma_A$-connected to make sense.

Let $F$ be the centered fundamental domain for the lattice $\Gamma_A$ (that is, a translated fundamental domain for $\Gamma_A$ that is centered at the origin).  Note that $F$ is congruent to $\mathbb{R}^m / \Gamma_A$, and that the Lebesgue measure of $F$ is equal to the Lebesgue measure of $T$.  Suppose that $D$ is $\Gamma_A$-connected.

Let
\begin{align*}
T_0 &= F \\
T_1 &= A^{-1}(T_0 + D) \\
    &= A^{-1}(F+D) \\
T_2 &= A^{-1}(T_1 + D) \\
    &= A^{-1}(A^{-1}(F+D) + D) \\
    &= A^{-2}F + A^{-2}D + A^{-1}D \\
    &\vdots \\
T_n &= A^{-n}F + A^{-n}D + \cdots + A^{-2}D + A^{-1}D.
\end{align*}

Multiplying $T_n$ by $A^n$ we have the following equality
\begin{align*}
A^nT_n &= F + D + \cdots + A^{n-2}D + A^{n-1}D \\
       &= F + D_n.
\end{align*}
By our definition of $F$, this implies that $A^nT_n = F + D_n$ is connected with no finite cut sets (that is, no finite subset of points such that, if we remove those points, the resulting set would be disconnected) if and only if $D_n$ is $\Gamma_A$-connected.  The lemma follows, since $A^nT_n$ is connected (or has a finite cut set) if and only if $T_n$ is connected (respectively, has a finite cut set).
\end{proof}

Lemma~\ref{lem_limitconnected} together with the above lemma yield the following corollary.
\begin{corollary}  Suppose that $\Gamma_A$ is $A$-invariant.  If the level sets $D_n$ are $\Gamma_A$-connected for all sufficiently large $n$, then $T(A,D)$ is connected.
\end{corollary}
The condition that $\Gamma_A$ is $A$-invariant is required for $D_n \subset \Gamma_A$ for every $n \geq 1$.

We can simplify the criterion that the $D_n$ are $\Gamma_A$-connected for all sufficiently large $n$ to a sufficient condition for connectedness of $T(A,D)$ that is only slightly less general, but significantly easier to check.  Let $F$ be the centered fundamental domain for $\Gamma_A$, and let $S_{AF}$ be the set of edge neighbours of the parallelepiped $AF$ in $\Gamma_A$.  That is, $S_{AF}$ consists of the neighbours $g$ in $\Gamma_A$ of the set $AF$ such that $AF \cap (AF+g)$ is not just a single point.  Note that, since $AF$ is a parallelepiped, $S_{AF}$ consists of $2m$ points,
\[ S_{AF} = \{ \pm g_1, \ldots, \pm g_{m} \}, \]
where if $\{ b_1, \ldots, b_m \}$ is a basis for $\Gamma_A$ consisting of neighbours of $F$, then we can set $g_i = Ab_i$ for each $i=1,\ldots,m$.  That is,
\[ S_{AF}^{+} := \{ g_1, \ldots, g_m \} \]
is a basis for $A(\Gamma_A)$.
\begin{figure}[h]
\begin{center}
\scalebox{0.6}{\includegraphics{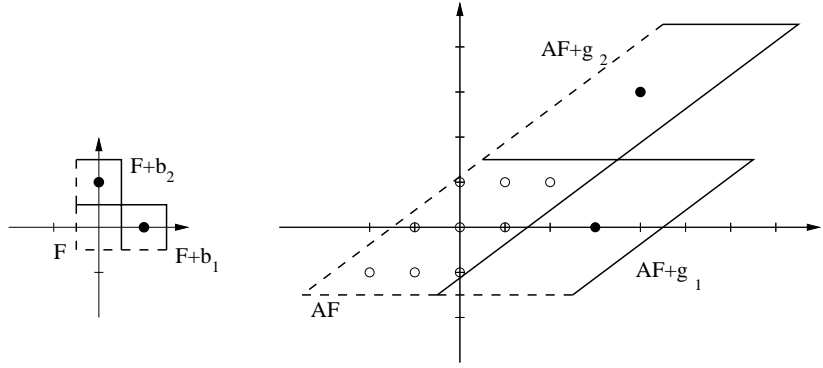}}
\end{center}
\caption{The edge neighbours of $AF$.  Here $\Gamma_A = \mathbb{Z}^2$; $b_1 = (1,0)$, $b_2 = (0,1)$, $g_1 = Ab_1 = (3,0)$, and $g_2 = Ab_2 = (4,3)$.}\label{fig_edge-nbrs_AF}
\end{figure}

\begin{theorem}\label{thm_goodsuffcond}
Suppose that $\Gamma_A$ is $A$-invariant.  Let $D$ be the centered canonical digit set for $A$.  If $(AF \cup (g+AF)) \cap \Gamma_A$ is $\Gamma_A$-connected for each $g \in S_{AF}$, then $T(A,D)$ is connected.
\end{theorem}
\begin{figure}[h]
\begin{center}
\subfigure[][The hypotheses are satisfied.]{\resizebox{!}{1.3in}{\includegraphics{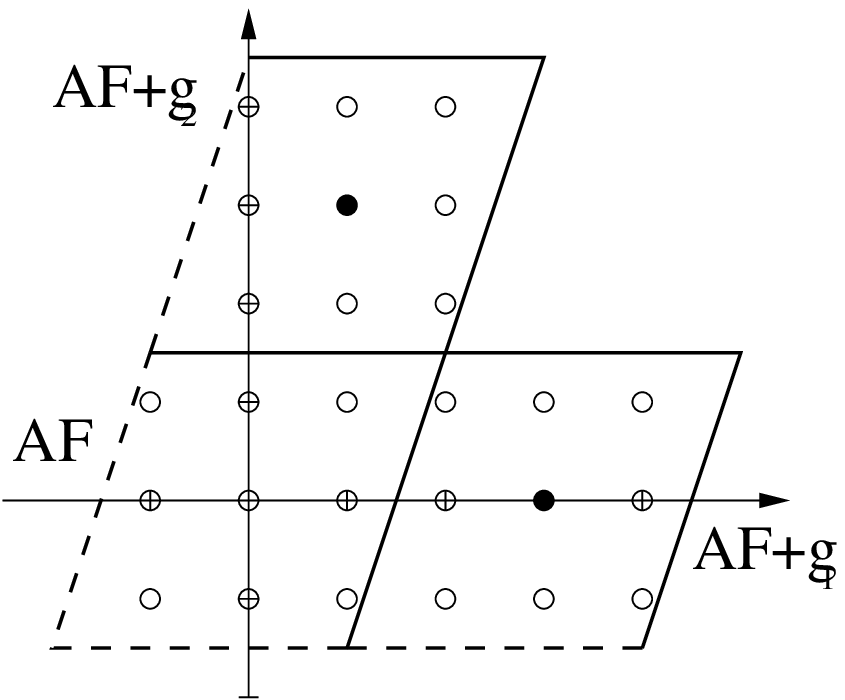}}}
\subfigure[][The hypotheses are \emph{not} satisfied.]{\resizebox{!}{1.3in}{\includegraphics{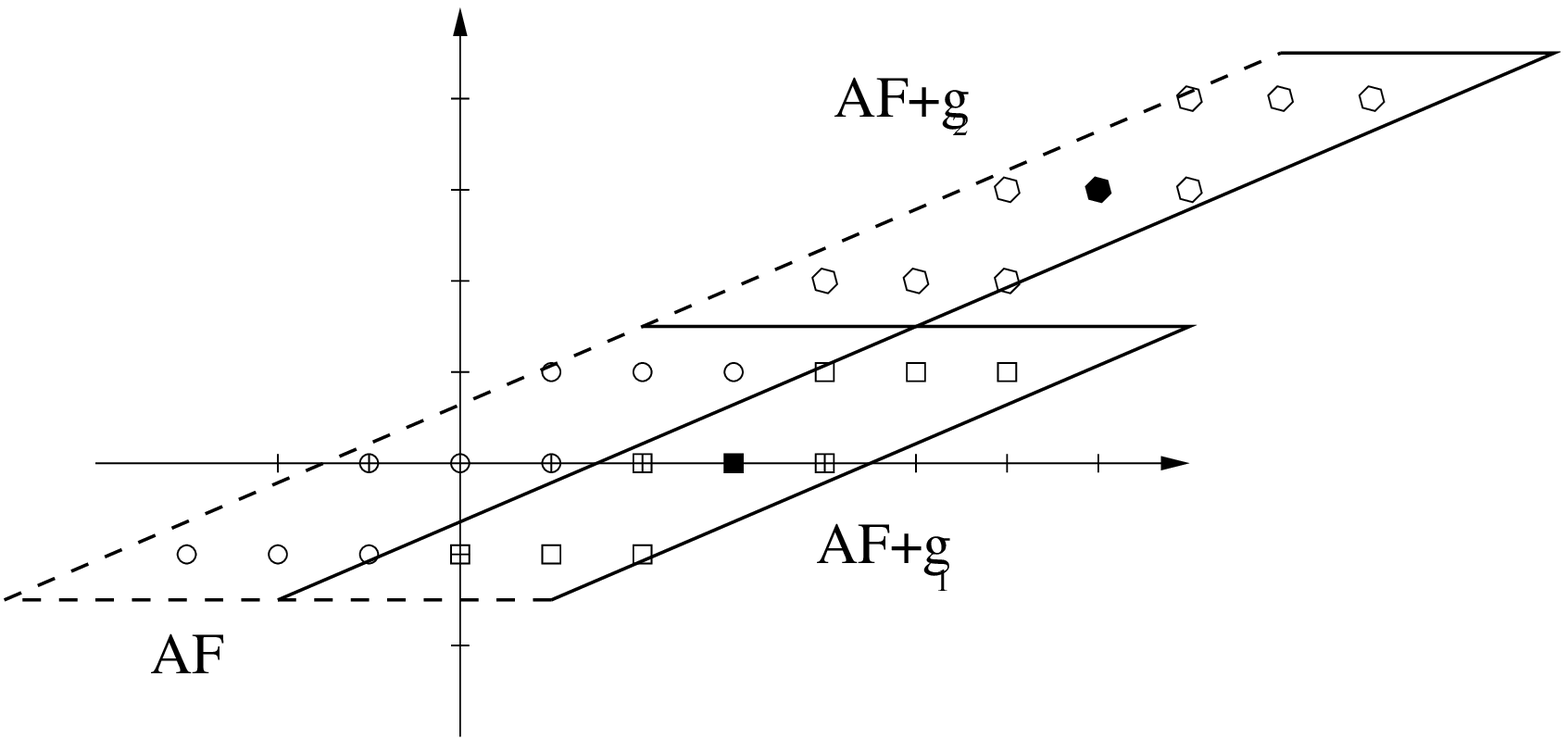}}}
\end{center}
\caption{Examples illustrating the hypothesis of Theorem~\ref{thm_goodsuffcond}.  In case (b), $(AF + (g_2 + AF)) \cap \mathbb{Z}^2$ is lattice-disconnected.}\label{fig_hypothesis}
\end{figure}
\begin{proof}
Suppose that $(AF \cup (g + S_{AF})) \cap \Gamma_A$ is $\Gamma_A$-connected for each $g \in S_{AF}$.  We will show inductively that $D_n$ is $\Gamma_A$-connected for each $n \geq 1$.  For the base case, note that $D_1 = D = AF \cap \Gamma_A$ by Lemma~\ref{lem_DcapGamma}.  Thus, by hypothesis, $D_1$ is $\Gamma_A$-connected.

For the inductive step, suppose that $D_{n-1}$ is $\Gamma_A$-connected.  Then $T_{n-1}$ is connected (with no finite cut set).  Recall from the proof of Lemma~\ref{lem_TnandDn} that $A^{n-1}T_{n-1} = D_{n-1} + F$.  Thus
\[ A^n T_{n-1} = AD_{n-1} + AF \]
is connected (with no finite cut sets).  Since $AF$ is a centered fundamental domain for $A(\Gamma_A)$, this implies that $AD_{n-1}$ is $A(\Gamma_{A})$-connected, equivalently, that $AD_{n-1}$ is $S_{AF}$-connected.  Since $AF + g$ is $\Gamma_A$-connected to $AF$ for each $g \in S_{AF}$ by hypothesis, we may thus conclude that $(AD_{n-1} + AF) \cap \Gamma_A$ is $\Gamma_A$-connected.  

Specifically, to find a connected path in $\Gamma_A$ between any two points $Ak_1 + d_{k_1}$ and $Ak_2 + d_{k_2}$ in $(AD_{n-1}+AF) \cap \Gamma_A$, we first find the $S_{AF}$-connected path $g_0 = Ak_1, \ldots, g_{r-1} = Ak_2$ between $k_1$ and $k_2$.  For each step $g_i$, we find the $\Gamma_A$-connected path from the origin $\mathbf{0}$ to $g_i$ (guaranteed to exist by the hypothesis that $(AF + (g+AF)) \cap \Gamma_A$ is $\Gamma_A$-connected for each $g \in S_{AF}$); denote this path by $d^{(i)}_1 = \mathbf{0}, \ldots, d^{(i)}_{s_i} = g_i$.  Set $\ell_{i+j} = d_{k_1} + g_i + d^{(i)}_j$ for $i = 0, \ldots, r-1$ and $j = 1, \ldots s_i$.  Then $\{ \ell_{i+j} \}$ forms a $\Gamma_A$-connected path from $Ak_1 + d_{k_1}$ to $Ak_2 + d_{k_2}$.  By hypothesis, $D$ is $\Gamma_A$-connected, so we append to this the $\Gamma_A$-connected path from $d_{k_1}$ to $d_{k_2}$ translated by $Ak_2$.  This gives the required path in $\Gamma_A$ connecting our two points.

Note also that
\[ (AD_{n-1} + AF) \cap \Gamma_A = AD_{n-1} + D = D_{n}. \]
Therefore we have shown that $D_n$ is $\Gamma_A$-connected, implying that $T_n$ is connected (with no finite cut sets).
\end{proof}

The most obvious (perhaps only) examples of dilation matrices $A$ that satisfy the hypotheses of Theorem~\ref{thm_goodsuffcond} are matrices that do not have too large a skew component, as illustrated in Figure~\ref{fig_hypothesis}.  If the eigenvalues of $A$ are of sufficiently large magnitude, so that all singular values $\sigma$ of $A$ satisfy $\sigma > 2$, then $D$ contains the standard basis for $\mathbb{Z}^m$, and thus $\Gamma_A = \mathbb{Z}^m$.

\section{Centered Canonical Digit Sets for $J$}

Let $A$ be a dilation matrix with rational eigenvalues.  Note that since $A \in M_{m}(\mathbb{Z})$, the eigenvalues are \emph{algebraic integers}, that is, roots of a monic polynomial with integer coefficients (the characteristic polynomial for $A$).  Algebraic integers have been studied extensively in connection with other algebraic and number theoretic questions.  For our purposes, the important result to note is that, if we denote the set of all algebraic integers by $\mathcal{A}$, then $\mathcal{A} \cap \mathbb{Q} = \mathbb{Z}$ (see, for example, Theorem 6.1.1 of \cite{ir1982}).  Thus we are considering dilation matrices $A$ with eigenvalues $\{ \lambda_1, \ldots, \lambda_r\}$ (for some $1 \leq r \leq m$), with $\lambda_i \in \mathbb{Z}$ for $i = 1, \ldots, r$.  Then the Jordan form $J = \mbox{diag}{\{ J_1, \ldots, J_r \}}$ is in $M_{m}(\mathbb{Z})$ as well.

Let $m_i$ be the size of the Jordan block corresponding to the eigenvalue $\lambda_i$, and $J_i \in M_{m_i}(\mathbb{Z})$ be the Jordan block corresponding to $\lambda_i$.  We will consider each Jordan block separately.  Set
\[ F_i := {\left( -\frac{1}{2},\frac{1}{2} \right]}^{m_i}. \]
Note that the hypotheses $\abs{\lambda_i} >1$ and $\lambda_i \in \mathbb{Z}$ together imply that $\abs{\lambda_i} \geq 2$.  Thus $J_iF_i$ is a parallelepiped in $\mathbb{R}^{m_i}$ with corners
\[ J_i {\left[ \begin{array}{c} \frac{\epsilon_1}{2} \\ \frac{\epsilon_2}{2} \\ \vdots \\ \frac{\epsilon_{m_i - 1}}{2} \\ \frac{\epsilon_{m_i}}{2} \end{array} \right]} = \frac{1}{2} {\left[ \begin{array}{c} \lambda_i \epsilon_1 + \epsilon_2 \\ \lambda_i \epsilon_2 + \epsilon_3 \\ \vdots \\ \lambda_i \epsilon_{m_i - 1} + \epsilon_{m_i} \\ \lambda_i \epsilon_{m_i} \end{array} \right]}, \]
where $\epsilon_j = \pm 1$ for $j = 1, \ldots, m_i$, and $J_iF_i$ is the convex hull of this set of points, excluding the faces where $\epsilon_j = -1$ for any $j$.  Examples with $\lambda_i = 2$ in dimensions $1$, $2$, and $3$ are shown in Figure~\ref{fig_JiFieven}.  Examples with $\lambda_i = 3$ in dimensions $1$, $2$, and $3$ are shown in Figure~\ref{fig_JiFiodd}.
\begin{figure}[h]
\begin{center}
\subfigure[][]{\resizebox{1.1in}{!}{\includegraphics{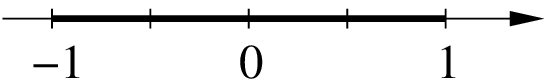}}}
\subfigure[][]{\resizebox{1.6in}{!}{\includegraphics{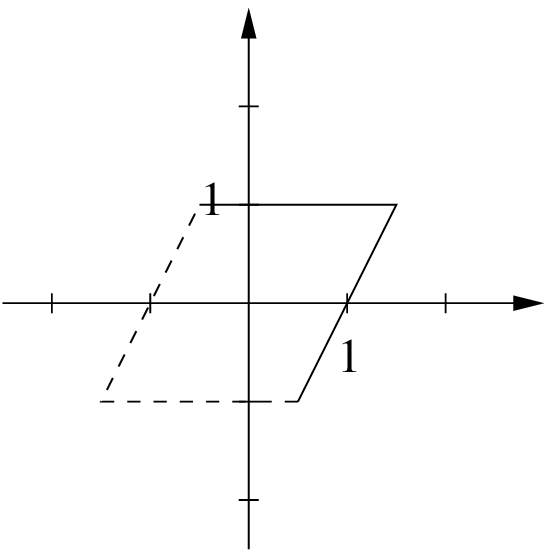}}}
\subfigure[][]{\resizebox{2.1in}{!}{\includegraphics{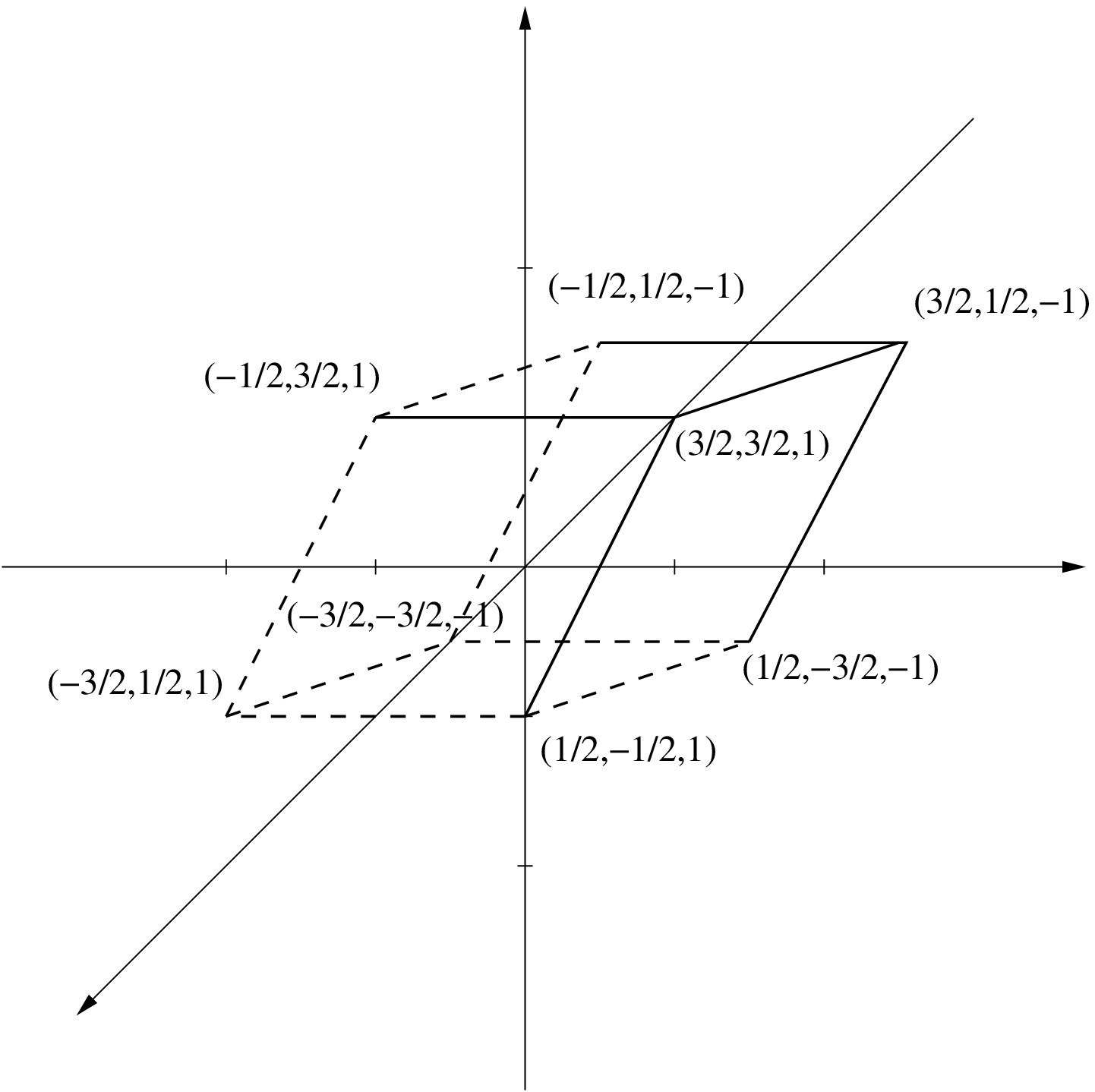}}}
\caption{The parallelepiped $J_iF_i$ for $\lambda_i = 2$ in dimensions (a) $m_i = 1$, (b) $m_i = 2$, and (c) $m_i = 3$.}\label{fig_JiFieven}
\end{center}
\end{figure}
\begin{figure}[h]
\begin{center}
\subfigure[][]{\resizebox{1.1in}{!}{\includegraphics{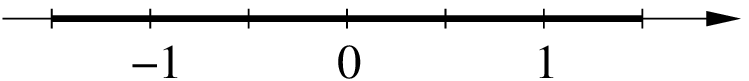}}}
\subfigure[][]{\resizebox{1.6in}{!}{\includegraphics{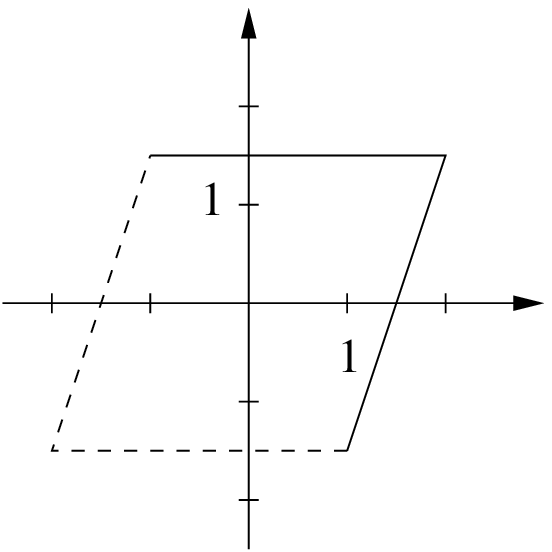}}}
\subfigure[][]{\resizebox{2.1in}{!}{\includegraphics{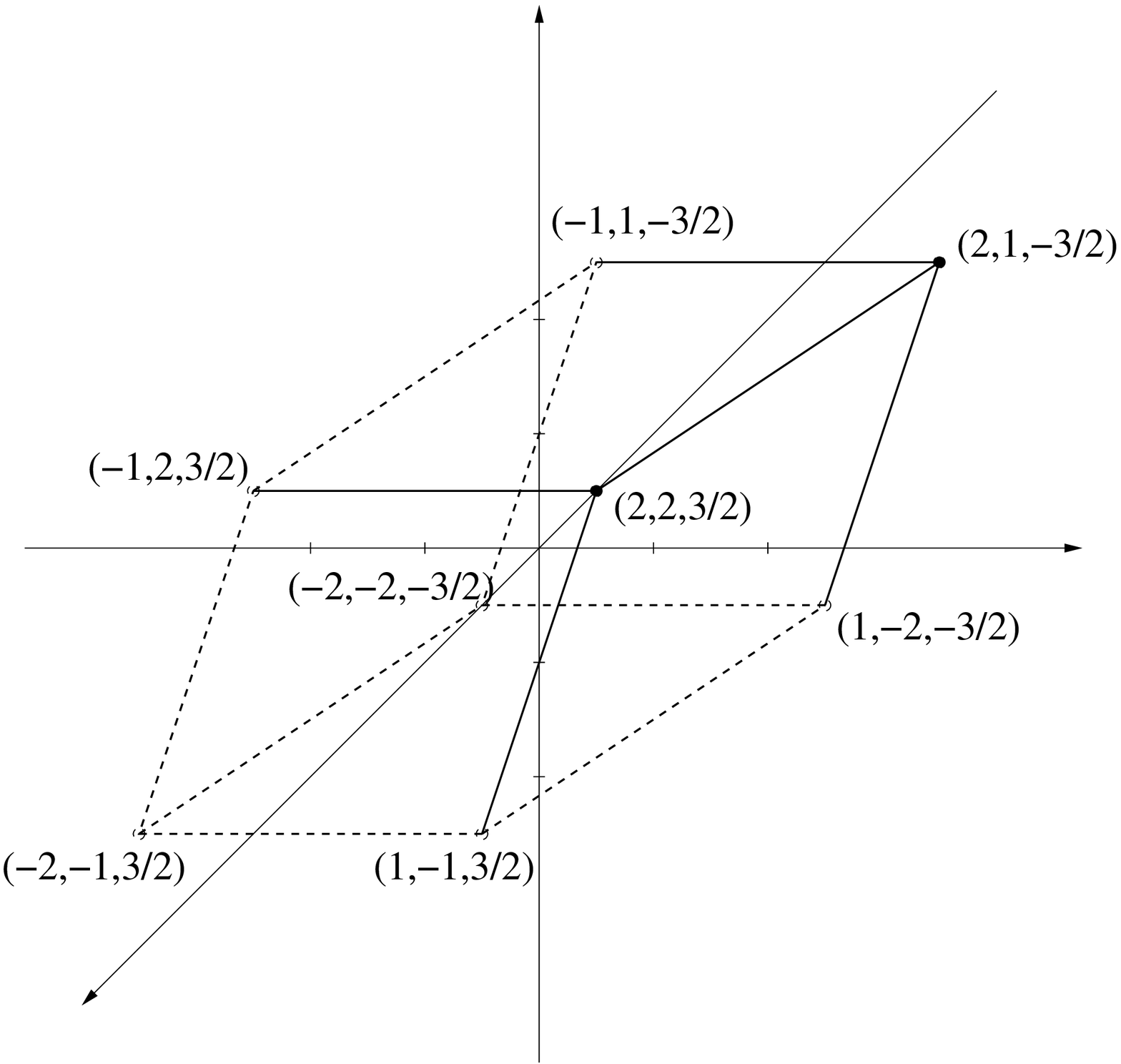}}}
\caption{The parallelepiped $J_iF_i$ for $\lambda_i = 3$ in dimensions (a) $m_i = 1$, (b) $m_i = 2$, and (c) $m_i = 3$.}\label{fig_JiFiodd}
\end{center}
\end{figure}

For any $\lambda_i$ and $m_i$, we note that $J_iF_i$ is contained in the cube
\[ C_{\mbox{outer}} := {\left( -\frac{\lambda_i}{2} - \frac{1}{2}, \frac{\lambda_i}{2} + \frac{1}{2} \right]}^{m_i - 1} \times {\left( -\frac{\lambda_i}{2}, \frac{\lambda_i}{2} \right]}, \]
and $J_iF_i$ contains the cube
\[ C_{\mbox{inner}} := {\left( -\frac{\lambda_i}{2} + \frac{1}{2}, \frac{\lambda_i}{2} - \frac{1}{2} \right]}^{m_i - 1} \times {\left( -\frac{\lambda_i}{2}, \frac{\lambda_i}{2} \right]}. \]
Thus
\[ C_{\mbox{inner}} \bigcap \mathbb{Z}^{m_i} \subset J_iF_i \bigcap \mathbb{Z}^{m_i} \subset C_{\mbox{outer}} \bigcap \mathbb{Z}^{m_i}. \]
Being a cube, $C_{\mbox{inner}} \cap \mathbb{Z}^{m_i}$ is lattice-connected.  Observe that the set of points \mbox{$(C_{\mbox{outer}} \cap \mathbb{Z}^{m_i}) \backslash (C_{\mbox{inner}} \cap \mathbb{Z}^{m_i})$} form a one-point-wide shell around $C_{\mbox{inner}} \cap \mathbb{Z}^{m_i}$ in all except the $m_i^{\mbox{th}}$-dimension.  Thus each point in this shell is lattice-adjacent to $C_{\mbox{inner}} \cap \mathbb{Z}^{m_i}$ with the possible exception of the corners, or extremal points.  That is, if $C_{\mbox{outer}}$ was truly the $m_i$-dimensional cube one unit larger than $C_{\mbox{inner}}$ in each direction, the corners would not be lattice-adjacent to $C_{\mbox{inner}}$, but would instead be at distance $2$ in the taxicab metric (see Figure~\ref{fig_NotOurCube}).
\begin{figure}[h]
\begin{center}
\resizebox{1.5in}{!}{\includegraphics{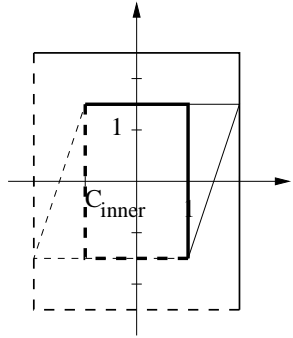}}
\caption{A cube around $C_{\mbox{inner}}$ extending one unit in each direction - \emph{not} $C_{\mbox{outer}}$!}\label{fig_NotOurCube}
\end{center}
\end{figure}
Instead, $C_{\mbox{outer}}$ coincides with $C_{\mbox{inner}}$ in the $m_i^{\mbox{th}}$-dimension, as shown in Figure~\ref{fig_IsOurCube}.
\begin{figure}[h]
\begin{center}
\subfigure[][{$\lambda_i = 2$}]{\resizebox{2in}{!}{\includegraphics{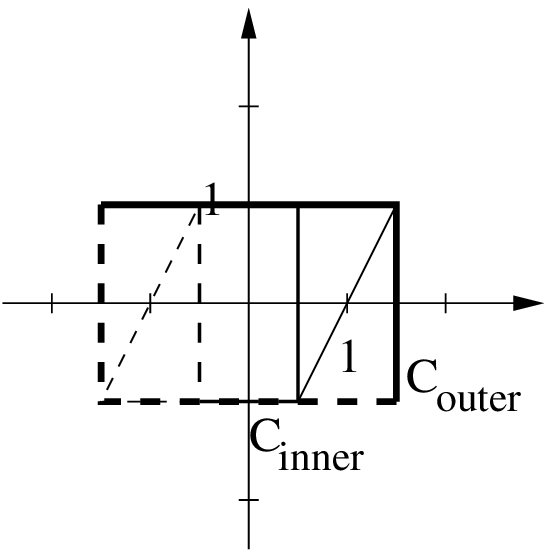}}}
\subfigure[][{$\lambda_i = 3$}]{\resizebox{2in}{!}{\includegraphics{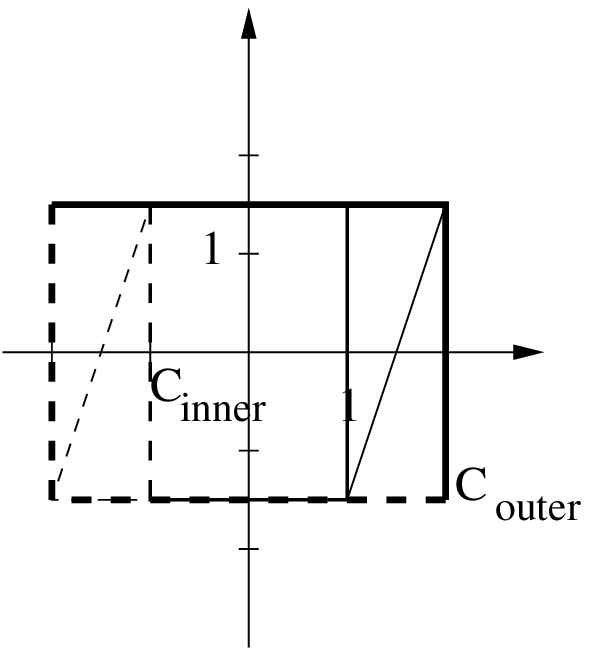}}}
\caption{$C_{\mbox{outer}}$ and $C_{\mbox{inner}}$ in dimension $m_i = 2$.}\label{fig_IsOurCube}
\end{center}
\end{figure}
The extremal points of $C_{\mbox{outer}}$ are thus of the form
\[ {\left[ \begin{array}{c} \epsilon_1 \frac{\lambda_i}{2} + \epsilon_2 \frac{1}{2} \\ \epsilon_2 \frac{\lambda_i}{2} + \epsilon_3 \frac{1}{2} \\ \vdots \\ \epsilon_{m_i - 1} \frac{\lambda_i}{2} + \epsilon_{m_i} \frac{1}{2} \\ \epsilon_{m_i} \frac{\lambda_i}{2} \end{array} \right]}, \quad \epsilon_1, \ldots, \epsilon_{m_i} = \pm 1. \]
If $\lambda_i$ is odd, the last coordinate will not be an integer.  If $\lambda_i$ is even, only the last coordinate will be an integer.  In both cases, the extremal points are not integer lattice points, and in fact $(C_{\mbox{outer}} \cap \mathbb{Z}^{m_i}) \backslash (C_{\mbox{inner}} \cap \mathbb{Z}^{m_i})$ consists only of points that are lattice-adjacent to \mbox{$C_{\mbox{inner}} \cap \mathbb{Z}^{m_i}$}.  $J_iF_i \cap \mathbb{Z}^{m_i}$ will include some points of this shell and not others, which will vary depending on whether $\lambda_i$ is even or odd; however in all cases we see that $J_iF_i \cap \mathbb{Z}^{m_i}$ is lattice-connected.

Set $D_{J_i} := J_iF_i \cap \mathbb{Z}^{m_i}$.  Note that $\Gamma_{J_i}$ contains a basis for $\mathbb{Z}^{m_i}$, thus $\Gamma_{J_i} = \mathbb{Z}^{m_i}$, and is thus a $J_i$-invariant lattice.  Applying the observation above together with Theorem~\ref{thm_goodsuffcond}), we obtain the following.
\begin{proposition} $T(J_i,D_{J_i})$ is connected.
\end{proposition}

Since $J = \oplus_{i} J_i$, $D_{J} := \prod_{i} D_{J_i}$ is a centered canonical digit set for $J$, and $T(J,D_J) = \prod_i T(J_i,D_{J_i})$.  Thus we conclude:
\begin{corollary} $T(J,D_J)$ is connected.
\end{corollary}
Also, taking $P \in M_m(\mathbb{Z})$ and setting $D_A := PD_J$, so that
$T(A,D_A) = PT(J,D_J)$:
\begin{corollary} Let $A$ be any dilation matrix with only rational eigenvalues.
There exists a digit set $D_A$ for $A$ such that $T(A,D_A)$ is connected.
\end{corollary}

\end{document}